\documentclass[a4paper,twoside, 11pt]{article}
\usepackage[margin=1in]{geometry}
\usepackage{amsmath,amsfonts,amsthm}
\usepackage{graphicx,cite,times}
\usepackage{enumitem,version,lipsum}
\usepackage{pdfpages,hyperref,fancyhdr}
\usepackage{tikz,xcolor,hyperref}
\definecolor{lime}{HTML}{A6CE39}
\DeclareRobustCommand{\orcidicon}{\hspace{-1mm}
	\begin{tikzpicture}
		\draw[lime, fill=lime](0,0) circle[radius=0.16] 
		node[white] {{\fontfamily{qag}\selectfont \tiny ID}};
		\draw[white, fill=white] (-0.0625,0.095) 
		circle [radius=0.007];
	\end{tikzpicture}
	\hspace{-1mm}}
\foreach \x in {A, ..., Z}{\expandafter\xdef\csname orcid\x\endcsname{\noexpand\href{https://orcid.org/\csname orcidauthor\x\endcsname}
		{\noexpand\orcidicon}}}

\newtheorem{theorem}{Theorem}[section]

\newtheorem{definition}{Definition}[section]

\numberwithin{equation}{section}
\Large

\begin{document}
	\title{On $k$-Mersenne and $k$-Mersenne-Lucas Octonions}
	\author{ Munesh Kumari$^{1}$\orcidA{}\footnote{E-mail: muneshnasir94@gmail.com }, Kalika Prasad$^{1}$\orcidB{}\footnote{E-mail: klkaprsd@gmail.com}, and 
	Hrishikesh Mahato$^{1}$\orcidD{}\footnote{E-mail: hrishikesh.mahato@cuj.ac.in} 
	\\\normalsize{$^{1}$\small Department of Mathematics, Central University of Jharkhand, India, 835205}}
	\date{\today}
	\maketitle
	\noindent\rule{15cm}{.15pt}
	\begin{abstract}
		This paper aims to introduce the $k$-Mersenne and $k$-Mersenne-Lucas octonions.  We give the closed form formulae for these octonions and obtain some well-known identities like Cassini's identity, d'Ocagne's identity, Catalan identity, Vajda's identity and generating functions of them.
		As a consequence $k=1$ yields all the above properties for Mersenne and Mersenne-Lucas octonions.
	\end{abstract}
	\noindent\rule{15cm}{.15pt}
	\section{Introduction}
	In 1843, W.R.Hamilton extended the concept of set of complex numbers to the set of quaternions denoted as $\mathbb{H}$. For $a,b,c,d \in \mathbb{R}$, a quaternion $q \in \mathbb{H}$ is of the form $q=a+bi+cj+dk$, where $i^2=j^2=k^2=ijk=-1$. 
	\\Inspired by Hamilton's work, J.T. Graves defined the concept of the octonions in 1843. Later, in 1845, A. Cayley also defined the octonions. The set of octonions is usually denoted by $\mathbb{O}$. With a natural basis $\{ e_0=1, e_1=i, e_2=j , e_3=k, e_4=e, e_5=ie, e_6=je, e_7=ke\}$, $\mathbb{O}$ forms an $8$-dimensional non-associative division algebra over $ \mathbb{R} $. 
	
	If an element $a \in \mathbb{O}$ then it takes the form, 
	\begin{eqnarray}\label{octonion}
		a=\sum_{r=0}^{7}a_re_r, \quad \text{where $a_r \in \mathbb{R}$.}
	\end{eqnarray}
	And, the conjugate of the octonion $a$ is given as
	\begin{eqnarray}\label{conjugate}
		\overline{a}=a_0-\sum_{r=1}^{7}a_re_r.
	\end{eqnarray}
	The norm of an octonion $a$ is given as 
	\begin{eqnarray}\label{octnorm}
		N(a)=\sqrt{a\bar{a}}=\sqrt{\bar{a}a}=\sqrt{\sum_{r=0}^{7}a_r^2}. 
	\end{eqnarray}	
	The basis $\{ e_0=1, e_1=i, e_2=j , e_3=k, e_4=e, e_5=ie, e_6=je, e_7=ke\}$ follows the multiplication rule\cite{tian2000matrix} given in the Table \ref{t1}.
	\begin{table}[h]
		\centering	
		\caption{The multiplication table for the basis of $\mathbb{O}$}
		\label{t1}
		\begin{tabular}{|c| c| c| c| c| c| c|c| c| c|}
			\hline 
			. & 1 & $ e_1 $  &  $ e_2 $ &  $ e_3 $ & $ e_4 $ & $ e_5 $ & $ e_6 $ & $ e_7 $  \\
			\hline
			1 & 1 & $ e_1 $  &  $ e_2 $ &  $ e_3 $ & $ e_4 $ & $ e_5 $ & $ e_6 $ & $ e_7 $  \\
			\hline
			$ e_1 $ & $ e_1 $ & $ -1$  &  $ e_3 $ &  $ -e_2 $ & $ e_5 $ & $ -e_4 $ & $ -e_7 $ & $ e_6 $  \\
			\hline
			$ e_2 $ & $ e_2 $ & $ -e_3 $  &  $ -1 $ &  $ e_1 $ & $ e_6 $ & $ e_7 $ & $ -e_4 $ & $ -e_5 $  \\
			\hline
			$ e_3 $ & $ e_3 $ & $ e_2 $  &  $ -e_1 $ &  $ -1 $ & $ e_7 $ & $ -e_6 $ & $ e_5 $ & $ -e_4 $  \\
			\hline
			$ e_4 $ & $ e_4 $ & $ -e_5 $  &  $ -e_6 $ &  $ -e_7 $ & $ -1 $ & $ e_1 $ & $ e_2 $ & $ e_3 $  \\
			\hline
			$ e_5 $ & $ e_5 $ & $ e_4 $  &  $ -e_7 $ &  $ e_6 $ & $ -e_1 $ & $ -1 $ & $ -e_3 $ & $ e_2 $  \\
			\hline
			$ e_6 $ & $ e_6 $ & $ e_7 $  &  $ e_4 $ &  $ -e_5 $ & $ -e_2 $ & $ e_3 $ & $ -1 $ & $ -e_1 $  \\
			\hline
			$ e_7 $ & $ e_7 $ & $ -e_6 $  &  $ e_5 $ &  $ e_4 $ & $ -e_3 $ & $ -e_2 $ & $ e_1 $ & $ -1 $  \\
			\hline	
		\end{tabular}
	\end{table}
	\\For more reading on the quaternions and octonions, reader referred to \cite{baez2002octonions,conway2003quaternions}.
	
	In recent years, recursive sequences are of great interest among the researchers. Study of recursive sequences in division algebra was firstly presented by Horadam\cite{horadam1963complex} where they introduced Fibonacci and Lucas quaternions. After that, many researchers extended this study to other recursive sequences like Pell, Pell-Lucas, Jacobsthal and Jacobsthal-Lucas, etc. (for example, see{\cite{ccimen2016pell,szynal2016note,tasci2017k,catarino2016modified}}).
	Octonions with Fibonacci and Lucas components were introduced by Akkus and Ke{\c{c}}ilioglu\cite{akkus2015split} and they studied their properties like Binet formula, generating function and some well-known identities. Recently A.D. Godse\cite{godase2019hyperbolic} studied the hyperbolic Octonions involving $k$-Fibonacci \& $k$-Lucas and {\"O}zkan et.al. \cite{ozkan2022jacobsthal} studied the hyperbolic Octonions with $k$-Jacobsthal \& $k$-Jacobsthal Lucas sequences.	
	Some recent work on octonions with the sequences like Pell, Pell-Lucas, Jacobsthal, Jacobsthal-Lucas, Mersenne, Horadam etc. can be seen in\cite{szynal2016pell,catarino2016modified,ccimen2017jacobsthal,KarataHalici,malini2021mer}.
	
	Motivated essentially by recent works on octonions with the components from a recursive sequence, here we are considering the generalized recursive sequences so-called the $k$-Mersenne sequence and the $k$-Mersenne-Lucas sequence, a generalization of the Mersenne sequence. Many papers are dedicated to Mersenne sequence and their generalizations (see, for example \cite{frontczak2020mersenne,Soykan_2021,chelgham2021k,kumari2021some}).
	Daşdemir and Göksal\cite{dasdemir2019gaussian} have defined Mersenne quaternions and obtained Binet's formula and generating function of them.
	\\The Mersenne sequence $\{M_{n}\}_{n\ge0}$ is defined\cite{catarino2016mersenne} by
	\begin{equation}\label{mersene}
		M_{0}=0, \quad M_{1}=1, \quad	M_{n+1} = 3M_{n} -2M_{n-1}, \quad n\ge1,\nonumber
	\end{equation}
	and the $k$-Mersenne sequence $\{M_{k,n}\}_{n\ge0}$ is defined\cite{uslu2017some} recursively by 
	\begin{equation}\label{k-mersene}
		M_{k,0}=0, \quad M_{k,1}=1,\quad M_{k,n+1} = 3kM_{k,n} -2M_{k,n-1}, \quad n\ge1.
	\end{equation}
	The Mersenne-Lucas sequence $\{m_{n}\}_{n\ge0}$is defined\cite{saba2021mersenne}  by 
	\begin{equation}\label{merlucas}
		 m_{0}=2, \quad m_{1}=3,\quad m_{n+1} = 3m_{n} -2m_{n-1}, \quad n\ge1,\nonumber
	\end{equation}
	and the $k$-Mersenne-Lucas sequence $\{m_{k,n}\}_{n\ge0}$ is defined recursively by 
	\begin{equation}\label{k-merslucas}
		m_{k,0}=2, \quad m_{k,1}=3k,\quad m_{k,n+1} = 3km_{k,n} -2m_{k,n-1}, \quad n\ge1. 
	\end{equation}
	The Binet formulae of the $k$-Mersenne and $k$-Mersenne-Lucas sequences are given, respectively, by
	\begin{equation}\label{Binet}
		M_{k,n}=\dfrac{\lambda_{1}^{n}-\lambda_{2}^{n}}{\lambda_{1}-\lambda_{2}}, \quad \text{and} \quad m_{k,n}=\lambda_{1}^{n}+\lambda_{2}^{n},
	\end{equation}
	where $\lambda_{1}=\dfrac{3k+\sqrt{9k^2-8}}{2}$ and $\lambda_{2}=\dfrac{3k-\sqrt{9k^2-8}}{2}$ are the roots of the characteristic equation $\lambda^2-3k\lambda+2=0$ associated with the above recurrence relations.
	
	Note that $ \lambda_{1} $ and $ \lambda_{2} $ have the following properties.
	\begin{eqnarray}
		\lambda_{1}+\lambda_{2}=3k, \quad \lambda_{1}\lambda_{2}=2,\quad \lambda_{1}-\lambda_{2}=\sqrt{9k^2-8}
	\end{eqnarray}
	and
	\begin{eqnarray}
		\dfrac{\lambda_{1}}{\lambda_{2}}=\dfrac{\lambda_{1}^2}{2},\quad \dfrac{\lambda_{2}}{\lambda_{1}}=\dfrac{\lambda_{2}^2}{2}.\nonumber
	\end{eqnarray}
\section{$k$-Mersenne Octonions}
	In this section, we define the $k$-Mersenne Octonions, obtain their closed form formula and present some well-known identities and properties of them.
	\begin{definition}\label{kmersenneoct}
		For $n\geq 0$, any $n^{th}$ $k$-Mersenne octonion $M\mathbb{O}_{k,n}$ is defined by the relation
		\begin{eqnarray}
			M\mathbb{O}_{k,n}=\sum_{r=0}^{7}M_{k,n+r}e_r, 
		\end{eqnarray}
		where $M_{k,n}$ is the $n^{th}$ $k$-Mersenne number.
	\end{definition}
	Using the Definition \ref{kmersenneoct} and equation \eqref{k-mersene}, after some elementary calculations, we get the following recurrence relation for the $k$-Mersenne octonions,
	\begin{eqnarray}\label{octrecursive}
		M\mathbb{O}_{k,n+1} = 3kM\mathbb{O}_{k,n} -2M\mathbb{O}_{k,n-1}. 
	\end{eqnarray}
	By the equation \eqref{conjugate}, the conjugate of the $k$-Mersenne octonion is 
	\begin{eqnarray}
		\overline{M\mathbb{O}}_{k,n}=M_{k,0}-\sum_{r=1}^{7}M_{k,n+r}e_r.
	\end{eqnarray}
	We should note that if we take $k=1$ in the expression \eqref{octrecursive}, we have the recursive formula for the $n^{th}$ Mersenne octonion $M\mathbb{O}_{n}$ given as
	\begin{eqnarray}\label{octMerrecursive}
		M\mathbb{O}_{n+1} = 3M\mathbb{O}_{n} -2M\mathbb{O}_{n-1}. \nonumber
	\end{eqnarray}
	\begin{theorem}
		For any integer $n\geq 0$, the norm of the $ n^{th} $ $k$-Mersenne octonion $M\mathbb{O}_{k,n}$ is 
		\begin{eqnarray}
			N(M\mathbb{O}_{k,n})=\sqrt{\dfrac{\lambda_{1}^{2n}(1+\lambda_{1}^2+...+\lambda_{1}^{14})+\lambda_{2}^{2n}(1+\lambda_{2}^2+...+\lambda_{2}^{14})-255.2^{n+1}}{9k^2-8}.}
		\end{eqnarray}
	\end{theorem}
	\begin{proof}
		By the eqn. \eqref{octnorm}, we have
		\begin{eqnarray}
			N(M\mathbb{O}_{k,n})^2 &=&\sum_{r=0}^{7}M_{k,n+r}^2\nonumber\\
			&=&\sum_{r=0}^{7}\left(\dfrac{\lambda_{1}^{n+r}-\lambda_{2}^{n+r}}{\lambda_{1}-\lambda_{2}}\right)^2\nonumber\\
			&=&\dfrac{\lambda_{1}^{2n}(1+\lambda_{1}^2+...+\lambda_{1}^{14})+\lambda_{2}^{2n}(1+\lambda_{2}^2+...+\lambda_{2}^{14})-255.2^{n+1}}{9k^2-8}.\nonumber
		\end{eqnarray} Thus, this completes the proof.
	\end{proof}
	\begin{theorem}
		The closed form formula of the $k$-Mersenne octonions is given as
		\begin{eqnarray}\label{binetoct}
			M\mathbb{O}_{k,n}=\dfrac{\alpha\lambda_{1}^{n}-\beta\lambda_{2}^n}{\sqrt{9k^2-8}},
		\end{eqnarray}
		where $\alpha=\sum_{r=0}^{7}\lambda_{1}^{r}e_r$ and $\beta=\sum_{r=0}^{7}\lambda_{2}^{r}e_r .$
	\end{theorem}
	\begin{proof}
		By using the Binet formula of $k$-Mersenne \eqref{Binet} in the Definition \ref{kmersenneoct}, we get
		\begin{eqnarray}
			M\mathbb{O}_{k,n} &=&\sum_{r=0}^{7}\left(\dfrac{\lambda_{1}^{n+r}-\lambda_{2}^{n+r}}{\lambda_{1}-\lambda_{2}}\right)e_r \nonumber\\
			&=&\dfrac{1}{\lambda_{1}-\lambda_{2}}\left(\lambda_{1}^{n}\sum_{r=0}^{7}\lambda_{1}^{r}e_r-\lambda_{2}^n\sum_{r=0}^{7}\lambda_{2}^{r}e_r\right) \nonumber\\
			&=& \dfrac{\alpha\lambda_{1}^{n}-\beta\lambda_{2}^n}{\sqrt{9k^2-8}} \nonumber,
		\end{eqnarray}
		where $\alpha=\sum_{r=0}^{7}\lambda_{1}^{r}e_r$ and $\beta=\sum_{r=0}^{7}\lambda_{2}^{r}e_r .$
	\end{proof}
	With the help of the above closed form formula, we obtain some properties of $k$-Mersenne octonions given in the following theorems. Throughout the paper, we use $\alpha=\sum_{r=0}^{7}\lambda_{1}^{r}e_r$ and $\beta=\sum_{r=0}^{7}\lambda_{2}^{r}e_r.$ Note that $ \mathbb{O} $ is a non-commutative algebra and hence $\alpha\beta \ne \beta\alpha.$
	\begin{theorem}[Catalan's Identity]\label{catalan}
		For $n,r\in \mathbb{N}$ such that $n \geq r$, we have
		\begin{enumerate}
			\item
				$M\mathbb{O}_{k,n+r}M\mathbb{O}_{k,n-r}-M\mathbb{O}_{k,n}^2= \dfrac{2^{n-r}[\alpha\beta(2^{r}-\lambda_{1}^{2r})+\beta\alpha(2^{r}-\lambda_{2}^{2r})]}{9k^2-8}, $
			\item 
				$M\mathbb{O}_{k,n-r}M\mathbb{O}_{k,n+r}-M\mathbb{O}_{k,n}^2=
				\dfrac{2^{n-r}[\alpha\beta(2^{r}-\lambda_{2}^{2r})+\beta\alpha(2^{r}-\lambda_{1}^{2r})]}{9k^2-8}. $
		\end{enumerate}	
	\end{theorem}
	\begin{proof}[Proof (1)]
		Using relation \eqref{binetoct} for $k$-Mersenne octonions, we can write
		\begin{eqnarray}
			M\mathbb{O}_{k,n+r}M\mathbb{O}_{k,n-r}-M\mathbb{O}_{k,n}^2
			&=&\left(\dfrac{\alpha\lambda_{1}^{n+r}-\beta\lambda_{2}^{n+r}}{\sqrt{9k^2-8}}\right)\left(\dfrac{\alpha\lambda_{1}^{n-r}-\beta\lambda_{2}^{n-r}}{\sqrt{9k^2-8}}\right)-\left(\dfrac{\alpha\lambda_{1}^{n}-\beta\lambda_{2}^{n}}{\sqrt{9k^2-8}}\right)^2\nonumber\\
			&=&\dfrac{\alpha\beta\lambda_{1}^{n}\lambda_{2}^{n}+\beta\alpha\lambda_{1}^{n}\lambda_{2}^{n}-\alpha\beta\lambda_{1}^{n+r}\lambda_{2}^{n-r}-\beta\alpha\lambda_{1}^{n-r}\lambda_{2}^{n+r}}{9k^2-8}\nonumber\\
			&=&\dfrac{2^{n}[\alpha\beta(1-\lambda_{1}^r\lambda_{2}^{-r})+\beta\alpha(1-\lambda_{1}^{-r}\lambda_{2}^{r})]}{9k^2-8}\nonumber\\
			&=&\dfrac{2^{n-r}[\alpha\beta(2^{r}-\lambda_{1}^{2r})+\beta\alpha(2^{r}-\lambda_{2}^{2r})]}{9k^2-8}.\nonumber
		\end{eqnarray}
		The proof of part (2) is similar to (1).		
	\end{proof}
	\begin{theorem}[Cassini's Identity]
		For $n\in \mathbb{N}$, we have
		\begin{enumerate}
			\item 
				$M\mathbb{O}_{k,n+1}M\mathbb{O}_{k,n-1}-M\mathbb{O}_{k,n}^2=
				\dfrac{2^{n-1}[\alpha\beta(2-\lambda_{1}^{2})+\beta\alpha(2-
				\lambda_{2}^{2})]}{9k^2-8}, $
			\item 
				$M\mathbb{O}_{k,n-1}M\mathbb{O}_{k,n+1}-M\mathbb{O}_{k,n}^2=
				\dfrac{2^{n-1}[\alpha\beta(2-\lambda_{2}^{2})+\beta\alpha(2-\lambda_{1}^{2})]}{9k^2-8}$.
		\end{enumerate}	
	\end{theorem}
	\begin{proof}
		For $r=1$ in the Catalan's identity given in the Theorem \ref{catalan}, we get the Cassini's identity.
	\end{proof}
	\begin{theorem}[d'Ocagne's Identity]\label{d'Ocagne}
		Let $n, r$ be any nonnegative integers, then the d'Ocagne's identity for $k$-Mersenne octonions is given by
		$$M\mathbb{O}_{k,r}M\mathbb{O}_{k,n+1}-M\mathbb{O}_{k,r+1}M\mathbb{O}_{k,n}= \dfrac{\alpha\beta\lambda_{1}^r\lambda_{2}^n-\beta\alpha\lambda_{1}^n\lambda_{2}^r}{\sqrt{9k^2-8}}.$$
	\end{theorem}
	\begin{proof}
		From Binet formula \eqref{binetoct}, we have
		\begin{eqnarray}
			M\mathbb{O}_{k,r}M\mathbb{O}_{k,n+1}-M\mathbb{O}_{k,r+1}M\mathbb{O}_{k,n}
			&=&\left(\dfrac{\alpha\lambda_{1}^{r}-\beta\lambda_{2}^{r}}{\sqrt{9k^2-8}}\right)\left(\dfrac{\alpha\lambda_{1}^{n+1}-\beta\lambda_{2}^{n+1}}{\sqrt{9k^2-8}}\right)\nonumber\\
			&-&\left(\dfrac{\alpha\lambda_{1}^{r+1}-\beta\lambda_{2}^{r+1}}{\sqrt{9k^2-8}}\right)\left(\dfrac{\alpha\lambda_{1}^{n}-\beta\lambda_{2}^{n}}{\sqrt{9k^2-8}}\right)\nonumber\\
			&=&\dfrac{\alpha\beta\lambda_{1}^{r+1}\lambda_{2}^{n}+\beta\alpha\lambda_{1}^{n}\lambda_{2}^{r+1}-\alpha\beta\lambda_{1}^{r}\lambda_{2}^{n+1}-\beta\alpha\lambda_{1}^{n+1}\lambda_{2}^{r}}{9k^2-8}\nonumber\\
			&=&\dfrac{\alpha\beta\lambda_{1}^r\lambda_{2}^n(\lambda_{1}-\lambda_{2})-\beta\alpha\lambda_{1}^n\lambda_{2}^r(\lambda_{1}-\lambda_{2})}{9k^2-8}\nonumber\\
			&=&\dfrac{\alpha\beta\lambda_{1}^r\lambda_{2}^n-\beta\alpha\lambda_{1}^n\lambda_{2}^r}{\sqrt{9k^2-8}}.\nonumber
		\end{eqnarray}
	As required.	
	\end{proof}
	\begin{theorem}[Vajda's Identity]\label{Vajda}
		Let $n,i$ $ \& j$ be any non-negative integers then we have
		\begin{eqnarray}
			M\mathbb{O}_{k,n+i}M\mathbb{O}_{k,n+j}-M\mathbb{O}_{k,n}M\mathbb{O}_{k,n+i+j}= \dfrac{2^nM_{k,i}[\beta\alpha\lambda_{1}^j-\alpha\beta\lambda_{2}^j]}{\sqrt{9k^2-8}}.\nonumber
		\end{eqnarray}		
	\end{theorem}
	\begin{proof}
		By using the Binet formula for the $k$-Mersenne octonions, we have
		\begin{eqnarray}
			M\mathbb{O}_{k,n+i}M\mathbb{O}_{k,n+j}-M\mathbb{O}_{k,n}M\mathbb{O}_{k,n+i+j}
			&=&\left(\dfrac{\alpha\lambda_{1}^{n+i}-\beta\lambda_{2}^{n+i}}{\sqrt{9k^2-8}}\right)\left(\dfrac{\alpha\lambda_{1}^{n+j}-\beta\lambda_{2}^{n+j}}{\sqrt{9k^2-8}}\right)\nonumber\\
			&-&\left(\dfrac{\alpha\lambda_{1}^{n}-\beta\lambda_{2}^{n}}{\sqrt{9k^2-8}}\right)\left(\dfrac{\alpha\lambda_{1}^{n+i+j}-\beta\lambda_{2}^{n+i+j}}{\sqrt{9k^2-8}}\right)\nonumber\\
			&=&\dfrac{\alpha\beta\lambda_{1}^{n}\lambda_{2}^{n+i+j}+\beta\alpha\lambda_{1}^{n+i+j}\lambda_{2}^{n}-\alpha\beta\lambda_{1}^{n+i}\lambda_{2}^{n+j}-\beta\alpha\lambda_{1}^{n+j}\lambda_{2}^{n+i}}{9k^2-8}\nonumber\\
			&=&\dfrac{(\lambda_{1}\lambda_{2})^n[\alpha\beta\lambda_{2}^j(\lambda_{2}^i-\lambda_{1}^i)+\beta\alpha\lambda_{1}^j(\lambda_{1}^i-\lambda_{2}^i)]}{9k^2-8}\nonumber\\
			&=&\dfrac{2^nM_{k,i}[\beta\alpha\lambda_{1}^j-\alpha\beta\lambda_{2}^j]}{\sqrt{9k^2-8}}.\nonumber
		\end{eqnarray}
	As required.
	\end{proof}
	\begin{theorem}
		The ordinary and exponential generating function for the $k$-Mersenne octonions are given, respectively, as
		\begin{enumerate}
			\item $ \sum_{n=0}^{\infty}M\mathbb{O}_{k,n}x^n=\dfrac{M\mathbb{O}_{k,0}+x\left(M\mathbb{O}_{k,1}-3kM\mathbb{O}_{k,0}\right)}{ 1-3x+2x^2}, $
			\item $ \sum_{n=0}^{\infty}\dfrac{M\mathbb{O}_{k,n}x^n}{n!}= \dfrac{\alpha e^{\lambda_{1}x}-\beta e^{\lambda_{2}x}}{\sqrt{9k^2-8}}. $
		\end{enumerate}
	\end{theorem}
	\begin{proof}[Proof (1)]
		Consider the $k$-Mersenne octonion sequence $\{ M\mathbb{O}_{k,n}\}_{n=0}^\infty$ then the ordinary generating function for this sequence is
		\begin{eqnarray}
			gM\mathbb{O}(x)=\sum_{n=0}^{\infty}M\mathbb{O}_{k,n}x^n.\nonumber
		\end{eqnarray}
		Now using the closed form formula\eqref{binetoct}, we obtain
		\begin{eqnarray}
			\sum_{n=0}^{\infty}M\mathbb{O}_{k,n}x^n
			&=&\sum_{n=0}^{\infty}\left(\dfrac{\alpha\lambda_{1}^{n}-\beta\lambda_{2}^n}{\sqrt{9k^2-8}}\right)x^n\nonumber\\
			&=&\dfrac{1}{\sqrt{9k^2-8}}\left[\alpha\sum_{n=0}^{\infty}(\lambda_{1}x)^n-\beta\sum_{n=0}^{\infty}(\lambda_{2}x)^n\right]\nonumber\\
			&=&\dfrac{1}{\sqrt{9k^2-8}}\left[\alpha\left(\dfrac{1}{1-\lambda_{1}x}\right)-\beta\left(\dfrac{1}{1-\lambda_{2}x}\right)\right] \nonumber\\
			&=&\dfrac{1}{\sqrt{9k^2-8}}\left[\dfrac{\left(\alpha-\beta\right)+x\left(\beta\lambda_{1}-\alpha\lambda_{2}\right)}{ 1-3kx+2x^2}\right] \nonumber\\
			&=&\dfrac{M\mathbb{O}_{k,0}+x\left(M\mathbb{O}_{k,1}-3kM\mathbb{O}_{k,0}\right)}{ 1-3kx+2x^2}. \nonumber
		\end{eqnarray}
		As required.
		\\Proof of (\textit{2}) is same as of (\textit{1}), so we omit it. 
	\end{proof}
	\begin{theorem}
		 For $k \ne 1$, the finite sum formula for $k$-Mersenne octonions is given by,
			\begin{eqnarray}
				\sum_{j=0}^{n}M\mathbb{O}_{k,j}=\dfrac{2M\mathbb{O}_{k,n}-M\mathbb{O}_{k,n+1}+M\mathbb{O}_{k,1}+M\mathbb{O}_{k,0}(1-3k)}{3(1-k)}.\nonumber
			\end{eqnarray}
	\end{theorem}
	\begin{proof}
		 Using the Binet formula, we can write
			\begin{eqnarray}
				\sum_{j=0}^{n}M\mathbb{O}_{k,j}
				&=&\sum_{j=0}^{n}\left(\dfrac{\alpha\lambda_{1}^{j}-\beta\lambda_{2}^j}{\sqrt{9k^2-8}}\right)\nonumber\\
				&=&\dfrac{1}{\sqrt{9k^2-8}}\left[\alpha\sum_{j=0}^{n}\lambda_{1}^{j}-\beta\sum_{j=0}^{n}\lambda_{2}^j\right]\nonumber\\
				&=&\dfrac{1}{\sqrt{9k^2-8}}\left[\alpha\left(\dfrac{\lambda_{1}^{n+1}-1}{\lambda_{1}-1}\right)-\beta\left(\dfrac{\lambda_{2}^{n+1}-1}{\lambda_{2}-1}\right)\right]\nonumber\\
				&=&\dfrac{1}{\sqrt{9k^2-8}}\left[\dfrac{\lambda_{1}\lambda_{2}(\alpha\lambda_{1}^n-\beta\lambda_{2}^n)-(\alpha\lambda_{1}^{n+1}-\beta\lambda_{2}^{n+1})+(\beta\lambda_{1}-\alpha\lambda_{2})+(\alpha-\beta)}{\lambda_{1}\lambda_{2}-(\lambda_{1}+\lambda_{2})+1}\right]
				\nonumber\\
				&=&\dfrac{2M\mathbb{O}_{k,n}-M\mathbb{O}_{k,n+1}+M\mathbb{O}_{k,1}+M\mathbb{O}_{k,0}(1-3k)}{3(1-k)}.\nonumber
			\end{eqnarray} As required.
	\end{proof}
	\section{$k$-Mersenne-Lucas Octonions}
	In this section, we introduce the $k$-Mersenne-Lucas Octonions and present some properties of them like the Binet formula, generating function and some well-known identities.
	\begin{definition}\label{kmerlucasoct}
		For $n\geq 0$, the $n^{th}$ $k$-Mersenne-Lucas octonion $m\mathbb{O}_{k,n}$ is defined as
		\begin{eqnarray}
			m\mathbb{O}_{k,n}=\sum_{r=0}^{7}m_{k,n+r}e_r, 
		\end{eqnarray}
		where $m_{k,n}$ is the $n^{th}$ $k$-Mersenne-Lucas number.
	\end{definition}
	Using expression \eqref{k-mersene} in the Definition \ref{kmerlucasoct} and after some elementary calculations, we get the  recurrence relation for the $k$-Mersenne-Lucas octonions given as
	\begin{eqnarray}\label{k-lucasoctrecursive}
		m\mathbb{O}_{k,n+1} = 3km\mathbb{O}_{k,n} -2m\mathbb{O}_{k,n-1}. 
	\end{eqnarray}
	Note that for $k=1$, we have the definition of the Mersenne-Lucas octonion given recursively by
	\begin{eqnarray}\label{lucasoctrecursive}
		m\mathbb{O}_{n+1} = 3m\mathbb{O}_{n} -2m\mathbb{O}_{n-1}. 
	\end{eqnarray}
	The conjugate of the $k$-Mersenne-Lucas octonion $m\mathbb{O}_{k,n}$ can be written as
	\begin{eqnarray}
		\overline{m\mathbb{O}}_{k,n}=m_{k,0}-\sum_{r=1}^{7}m_{k,n+r}e_r.
	\end{eqnarray}
	\begin{theorem}
		For $n\geq 0$, the norm of the $ n^{th} $ $k$-Mersenne-Lucas octonion $m\mathbb{O}_{k,n}$ is 
		\begin{eqnarray}
			N(m\mathbb{O}_{k,n})=\sqrt{\lambda_{1}^{2n}(1+\lambda_{1}^2+...+\lambda_{1}^{14})+\lambda_{2}^{2n}(1+\lambda_{2}^2+...+\lambda_{2}^{14})+255.2^{n+1}}.
		\end{eqnarray}
	\end{theorem}
	\begin{proof}
		By the definition of norm, we have
		\begin{eqnarray}
			N(m\mathbb{O}_{k,n})^2&=&\sum_{r=0}^{7}m_{k,n+r}^2\nonumber\\
			&=&\sum_{r=0}^{7}\left(\lambda_{1}^{n+r}+\lambda_{2}^{n+r}\right)^2\nonumber\\
			&=&\lambda_{1}^{2n}(1+\lambda_{1}^2+...+\lambda_{1}^{14})+\lambda_{2}^{2n}(1+\lambda_{2}^2+...+\lambda_{2}^{14})+255.2^{n+1}.\nonumber
		\end{eqnarray} As required.
	\end{proof}
	\begin{theorem}
		The closed form formula of the $k$-Mersenne-Lucas octonions is given as
		\begin{eqnarray}\label{binetoctk-lucas}
			m\mathbb{O}_{k,n}=\alpha\lambda_{1}^{n}+\beta\lambda_{2}^n,
		\end{eqnarray}
		where $\alpha=\sum_{r=0}^{7}\lambda_{1}^{r}e_r$ and $\beta=\sum_{r=0}^{7}\lambda_{2}^{r}e_r .$
	\end{theorem}
	\begin{proof} 
		By using the Binet formula for $k$-Mersenne-Lucas in the Definition \ref{kmerlucasoct}, we get
		\begin{eqnarray}
			m\mathbb{O}_{k,n} &=&\sum_{r=0}^{7}\left(\lambda_{1}^{n+r}+\lambda_{2}^{n+r}\right)e_r \nonumber\\
			&=&\left(\lambda_{1}^{n}\sum_{r=0}^{7}\lambda_{1}^{r}e_r+\lambda_{2}^n\sum_{r=0}^{7}\lambda_{2}^{r}e_r\right) \nonumber\\
			&=& \alpha\lambda_{1}^{n}+\beta\lambda_{2}^n \nonumber,
		\end{eqnarray}
		where $\alpha=\sum_{r=0}^{7}\lambda_{1}^{r}e_r$ and $\beta=\sum_{r=0}^{7}\lambda_{2}^{r}e_r .$
	\end{proof}
	\begin{theorem}[Catalan's Identity]\label{catalanlucas}
		For $n,r\in \mathbb{N}$ such that $n \geq r$, we have
		\begin{enumerate}
			\item
			$m\mathbb{O}_{k,n+r}m\mathbb{O}_{k,n-r}-m\mathbb{O}_{k,n}^2= 2^{n-r}[\alpha\beta(\lambda_{1}^{2r}-2^{r})+\beta\alpha(\lambda_{2}^{2r}-2^{r})], $
			\item 
			$m\mathbb{O}_{k,n-r}m\mathbb{O}_{k,n+r}-m\mathbb{O}_{k,n}^2=
			2^{n-r}[\alpha\beta(\lambda_{2}^{2r}-2^{r})+\beta\alpha(\lambda_{1}^{2r}-2^{r})]. $
		\end{enumerate}	
	\end{theorem}
	\begin{proof}
		Using the Binet formula for $k$-Mersenne-Lucas octonions, we write
		\begin{eqnarray}
			m\mathbb{O}_{k,n+r}m\mathbb{O}_{k,n-r}-m\mathbb{O}_{k,n}^2
			&=&\left(\alpha\lambda_{1}^{n+r}+\beta\lambda_{2}^{n+r}\right)\left(\alpha\lambda_{1}^{n-r}+\beta\lambda_{2}^{n-r}\right)-\left(\alpha\lambda_{1}^{n}+\beta\lambda_{2}^{n}\right)^2\nonumber\\
			&=&\alpha\beta\lambda_{1}^{n+r}\lambda_{2}^{n-r}+\beta\alpha\lambda_{1}^{n-r}\lambda_{2}^{n+r}-\alpha\beta\lambda_{1}^{n}\lambda_{2}^{n}-\beta\alpha\lambda_{1}^{n}\lambda_{2}^{n}\nonumber\\
			&=&2^{n}[\alpha\beta(\lambda_{1}^r\lambda_{2}^{-r}-1)+\beta\alpha(\lambda_{1}^{-r}\lambda_{2}^{r})-1]\nonumber\\
			&=&2^{n-r}[\alpha\beta(\lambda_{1}^{2r}-2^{r})+\beta\alpha(\lambda_{2}^{2r}-2^{r})].\nonumber
		\end{eqnarray}
		By a similar argument, (2) can be proved so we omit it. 		
	\end{proof}
	\begin{theorem}[Cassini's Identity]
		For $n\in \mathbb{N}$, we have
		\begin{enumerate}
			\item 
			$m\mathbb{O}_{k,n+1}m\mathbb{O}_{k,n-1}-m\mathbb{O}_{k,n}^2=
			2^{n-1}[\alpha\beta(\lambda_{1}^{2}-2)+\beta\alpha(
				\lambda_{2}^{2}-2)], $
			\item 
			$m\mathbb{O}_{k,n-1}m\mathbb{O}_{k,n+1}-m\mathbb{O}_{k,n}^2=
			2^{n-1}[\alpha\beta(\lambda_{2}^{2}-2)+\beta\alpha(\lambda_{1}^{2}-2)].$
		\end{enumerate}	
	\end{theorem}
	\begin{proof}
		The results can be established by substituting $r=1$ in the Catalan's identity given in the Theorem \ref{catalanlucas}.
	\end{proof}
	\begin{theorem}[d'Ocagne's Identity]\label{d'Ocagnelucas}
		Let $n, r$ be any nonnegative integers, then d'Ocagne's identity for $k$-Mersenne-Lucas octonions is given by
		$$m\mathbb{O}_{k,r}m\mathbb{O}_{k,n+1}-m\mathbb{O}_{k,r+1}m\mathbb{O}_{k,n}= (\sqrt{9k^2-8})(\beta\alpha\lambda_{1}^n\lambda_{2}^r-\alpha\beta\lambda_{1}^r\lambda_{2}^n).$$
	\end{theorem}
	\begin{proof}
		By the Binet formula \eqref{binetoct}, we have
		\begin{eqnarray}
			m\mathbb{O}_{k,r}m\mathbb{O}_{k,n+1}-m\mathbb{O}_{k,r+1}m\mathbb{O}_{k,n}=
			&=&\left(\alpha\lambda_{1}^{r}+\beta\lambda_{2}^{r}\right)\left(\alpha\lambda_{1}^{n+1}+\beta\lambda_{2}^{n+1}\right)
			-\left(\alpha\lambda_{1}^{r+1}+\beta\lambda_{2}^{r+1}\right)\left(\alpha\lambda_{1}^{n}+\beta\lambda_{2}^{n}\right)\nonumber\\
			&=&\alpha\beta\lambda_{1}^{r}\lambda_{2}^{n+1}+\beta\alpha\lambda_{1}^{n+1}\lambda_{2}^{r}-\alpha\beta\lambda_{1}^{r+1}\lambda_{2}^{n}-\beta\alpha\lambda_{1}^{n}\lambda_{2}^{r+1}\nonumber\\
			&=&\alpha\beta\lambda_{1}^r\lambda_{2}^n(\lambda_{2}-\lambda_{1})+\beta\alpha\lambda_{1}^n\lambda_{2}^r(\lambda_{1}-\lambda_{2})\nonumber\\
			&=&(\sqrt{9k^2-8})(\beta\alpha\lambda_{1}^n\lambda_{2}^r-\alpha\beta\lambda_{1}^r\lambda_{2}^n).\nonumber
		\end{eqnarray}
		This completes the proof.	
	\end{proof}
	\begin{theorem}[Vajda Identity]\label{Vajdalucas}
		Let $n,i$ $ \& j$ be any non-negative integers then we have
		\begin{eqnarray}
			m\mathbb{O}_{k,n+i}m\mathbb{O}_{k,n+j}-m\mathbb{O}_{k,n}m\mathbb{O}_{k,n+i+j}= 2^nM_{k,i}(\sqrt{9k^2-8})(\alpha\beta\lambda_{2}^j-\beta\alpha\lambda_{1}^j).\nonumber
		\end{eqnarray}		
	\end{theorem}
	\begin{proof}
		By the Binet formula \eqref{binetoct}, we have
		\begin{eqnarray}
			m\mathbb{O}_{k,n+i}m\mathbb{O}_{k,n+j}-m\mathbb{O}_{k,n}m\mathbb{O}_{k,n+i+j}
			&=&\left(\alpha\lambda_{1}^{n+i}+\beta\lambda_{2}^{n+i}\right)\left(\alpha\lambda_{1}^{n+j}+\beta\lambda_{2}^{n+j}\right)\nonumber\\
			&-&\left(\alpha\lambda_{1}^{n}+\beta\lambda_{2}^{n}\right)\left(\alpha\lambda_{1}^{n+i+j}+\beta\lambda_{2}^{n+i+j}\right)\nonumber\\
			&=&\alpha\beta\lambda_{1}^{n+i}\lambda_{2}^{n+j}+\beta\alpha\lambda_{1}^{n+j}\lambda_{2}^{n+i}-\alpha\beta\lambda_{1}^{n}\lambda_{2}^{n+i+j}-\beta\alpha\lambda_{1}^{n+i+j}\lambda_{2}^{n}\nonumber\\
			&=&(\lambda_{1}\lambda_{2})^n[\alpha\beta\lambda_{2}^j(\lambda_{1}^i-\lambda_{2}^i)+\beta\alpha\lambda_{1}^j(\lambda_{2}^i-\lambda_{1}^i)]\nonumber\\
			&=&2^nM_{k,i}(\sqrt{9k^2-8})(\alpha\beta\lambda_{2}^j-\beta\alpha\lambda_{1}^j).\nonumber
		\end{eqnarray}
		As required.		
	\end{proof}
	\begin{theorem}
		The ordinary and exponential generating function for the $k$-Mersenne-Lucas octonions are given, respectively, as
		\begin{enumerate}
			\item $ \sum_{n=0}^{\infty}m\mathbb{O}_{k,n}x^n=\dfrac{m\mathbb{O}_{k,0}+x\left(m\mathbb{O}_{k,1}-3km\mathbb{O}_{k,0}\right)}{ 1-3kx+2x^2}, $
			\item $ \sum_{n=0}^{\infty}\dfrac{m\mathbb{O}_{k,n}x^n}{n!}= \alpha e^{\lambda_{1}x}+\beta e^{\lambda_{2}x}. $
		\end{enumerate}
	\end{theorem}
	\begin{proof}
		Consider the $k$-Mersenne-Lucas octonion sequence $\{ m\mathbb{O}_{k,n}\}_{n=0}^\infty$ then the ordinary generating function for this sequence is
		\begin{eqnarray}
			gm\mathbb{O}(x)=\sum_{n=0}^{\infty}m\mathbb{O}_{k,n}x^n.\nonumber
		\end{eqnarray}
		Now using the closed form formula \eqref{binetoct}, we obtain
		\begin{eqnarray}
			\sum_{n=0}^{\infty}m\mathbb{O}_{k,n}x^n
			&=&\sum_{n=0}^{\infty}\left(\alpha\lambda_{1}^{n}+\beta\lambda_{2}^n\right)x^n\nonumber\\
			&=&\alpha\sum_{n=0}^{\infty}(\lambda_{1}x)^n+\beta\sum_{n=0}^{\infty}(\lambda_{2}x)^n\nonumber\\
			&=&\alpha\left(\dfrac{1}{1-\lambda_{1}x}\right)+\beta\left(\dfrac{1}{1-\lambda_{2}x}\right) \nonumber\\
			&=&\dfrac{\left(\alpha+\beta\right)-x\left(\beta\lambda_{1}+\alpha\lambda_{2}\right)}{ 1-3kx+2x^2} \nonumber\\
			&=&\dfrac{m\mathbb{O}_{k,0}+x\left(m\mathbb{O}_{k,1}-3km\mathbb{O}_{k,0}\right)}{ 1-3kx+2x^2}. \nonumber
		\end{eqnarray}
		As required.
		\\Proof of (\textit{2}) is same as of (\textit{1}), so we omit it. 
	\end{proof}
	\begin{theorem}
		 For $k \ne 1$, the finite sum formula for $k$-Mersenne-Lucas octonions is given by
			\begin{eqnarray}
				\sum_{j=0}^{n}m\mathbb{O}_{k,j}=\dfrac{2m\mathbb{O}_{k,n}-m\mathbb{O}_{k,n+1}+m\mathbb{O}_{k,1}+m\mathbb{O}_{k,0}(1-3k)}{3(1-k)}.
				\nonumber
			\end{eqnarray}
	\end{theorem}
	\begin{proof}
		Using the Binet formula for $k$-Mersenne-Lucas octonions, we write
		\begin{eqnarray}
			\sum_{j=0}^{n}m\mathbb{O}_{k,j}
			&=&\sum_{j=0}^{n}\left(\alpha\lambda_{1}^{j}+\beta\lambda_{2}^j\right)\nonumber\\
			&=&\alpha\sum_{j=0}^{n}\lambda_{1}^{j}+\beta\sum_{j=0}^{n}\lambda_{2}^j\nonumber\\
			&=&\alpha\left(\dfrac{\lambda_{1}^{n+1}-1}{\lambda_{1}-1}\right)+\beta\left(\dfrac{\lambda_{2}^{n+1}-1}{\lambda_{2}-1}\right)\nonumber\\
			&=&\dfrac{\lambda_{1}\lambda_{2}(\alpha\lambda_{1}^n+\beta\lambda_{2}^n)-(\alpha\lambda_{1}^{n+1}+\beta\lambda_{2}^{n+1})-(\alpha\lambda_{2}+\beta\lambda_{1})+(\alpha+\beta)}{\lambda_{1}\lambda_{2}-(\lambda_{1}+\lambda_{2})+1}\nonumber\\
			&=&\dfrac{2m\mathbb{O}_{k,n}-m\mathbb{O}_{k,n+1}+m\mathbb{O}_{k,1}+m\mathbb{O}_{k,0}(1-3k)}{3(1-k)}.\nonumber
		\end{eqnarray} As required.
	\end{proof}
\section{Mersenne and Mersenne-Lucas Octonions}
	We should note that for $k=1$ in the $k$-Mersenne and $k$-Mersenne-Lucas octonions, we get Mersenne and Mersenne-Lucas octonions.	
	In this section, we present the above properties for $k=1$ i.e the Mersenne and Mersenne-Lucas octonions some of which are listed in \cite{malini2021mer}.
	\begin{theorem}
		For $n\geq 0$, norms for the the $n^{th} $ Mersenne and Mersenne-Lucas octonion are given, respectively, by
		\begin{enumerate}
			\item  $N(M\mathbb{O}_n)=\sqrt{21845.2^{2n}-510.2^n+8}.$
			\item $N(m\mathbb{O}_n)=\sqrt{21845.2^{2n}+510.2^n+8}.$
		\end{enumerate}
	\end{theorem}
	\begin{theorem}
		The Binet formulae of the Mersenne and Mersenne-Lucas octonions are  given, respectively by
		\begin{enumerate}
			\item $ M\mathbb{O}_{n}=\alpha2^{n}-\beta. $
			\item $ m\mathbb{O}_{n}=\alpha2^{n}+\beta. $
		\end{enumerate}
	\end{theorem}
	\begin{theorem}[Catalan's Identities]
			For $n,r\in \mathbb{N}$ such that $n \geq r$, we have
			\begin{enumerate}
				\item
				$M\mathbb{O}_{n+r}M\mathbb{O}_{n-r}-M\mathbb{O}_{n}^2= 2^{n}[\alpha\beta(1-2^r)+\beta\alpha(1-2^{-r})]. $
				\item 
				$M\mathbb{O}_{n-r}M\mathbb{O}_{n+r}-M\mathbb{O}_{n}^2=
				2^{n}[\alpha\beta(1-2^{-r})+\beta\alpha(1-2^r)]. $
				\item
				$m\mathbb{O}_{n+r}m\mathbb{O}_{n-r}-m\mathbb{O}_{n}^2= 2^{n}[\alpha\beta(2^{r}-1)+\beta\alpha(2^{-r}-1)]. $
				\item 
				$m\mathbb{O}_{n-r}m\mathbb{O}_{n+r}-m\mathbb{O}_{n}^2=
				2^{n}[\alpha\beta(2^{-r}-1)+\beta\alpha(2^{r}-1)]. $
			\end{enumerate}
	\end{theorem}
	\begin{theorem}[Cassini's Identities]
		For $n\in \mathbb{N}$, we have
		\begin{enumerate}
			\item $ M\mathbb{O}_{n+1}M\mathbb{O}_{n-1}-M\mathbb{O}_{n}^2=2^{n-1}(\beta\alpha-2\alpha\beta). $
			\item $ M\mathbb{O}_{n-1}M\mathbb{O}_{n+1}-M\mathbb{O}_{n}^2=2^{n-1}(\alpha\beta-2\beta\alpha). $
			\item 
			$m\mathbb{O}_{n+1}m\mathbb{O}_{n-1}-m\mathbb{O}_{n}^2=
			2^{n}[2\alpha\beta-\beta\alpha]. $
			\item 
			$m\mathbb{O}_{n-1}m\mathbb{O}_{n+1}-m\mathbb{O}_{n}^2=
			2^{n}[2\beta\alpha-\alpha\beta].$
		\end{enumerate}
	\end{theorem}
	\begin{theorem}[d'Ocagne's Identities]
		Let $n,r$ be any nonnegative integers, then we have,
		\begin{enumerate}
			\item
			$M\mathbb{O}_{r}M\mathbb{O}_{n+1}-M\mathbb{O}_{r+1}M\mathbb{O}_{n}= \alpha\beta2^r-\beta\alpha2^n.$
			\item
			$m\mathbb{O}_{r}m\mathbb{O}_{n+1}-m\mathbb{O}_{r+1}m\mathbb{O}_{n}= \beta\alpha2^n-\alpha\beta2^r.$
		\end{enumerate}	
	\end{theorem}
	\begin{theorem}[Vajda's Identity]
		Let $n,i$ $ \& j$ be any non-negative integers then we have
		\begin{enumerate}
			\item 	$ M\mathbb{O}_{k,n+i}M\mathbb{O}_{k,n+j}-M\mathbb{O}_{k,n}M\mathbb{O}_{k,n+i+j}= 2^nM_{i}[\beta\alpha2^j-\alpha\beta]. $
			\item $ m\mathbb{O}_{k,n+i}m\mathbb{O}_{k,n+j}-m\mathbb{O}_{k,n}m\mathbb{O}_{k,n+i+j}= 2^nM_{i}(\alpha\beta-\beta\alpha2^j). $
		\end{enumerate}	
	\end{theorem}
	\begin{theorem}
		The ordinary and exponential generating functions for the Mersenne and Mersenne-Lucas octonions are given as, respectively,
		\begin{enumerate}
			\item $ \sum_{n=0}^{\infty}M\mathbb{O}_{n}x^n=\dfrac{M\mathbb{O}_0+x\left(M\mathbb{O}_1-3M
				\mathbb{O}_0\right)}{ 1-3x+2x^2}. $
			\item $ \sum_{n=0}^{\infty}m\mathbb{O}_{n}x^n=\dfrac{m\mathbb{O}_{0}+x\left(m\mathbb{O}_{1}-3m\mathbb{O}_{0}\right)}{ 1-3x+2x^2}. $
			\item $ \sum_{n=0}^{\infty}\dfrac{M\mathbb{O}_{n}x^n}{n!}= \alpha e^{2x}-\beta e^{x}. $
			\item $ \sum_{n=0}^{\infty}\dfrac{m\mathbb{O}_{n}x^n}{n!}= \alpha e^{2x}+\beta e^{x}. $
		\end{enumerate}
	\end{theorem}
	\begin{theorem}
		The finite sum formulae for the Mersenne and Mersenne-Lucas octonions are given by, respectively,
		\begin{enumerate}
			\item $ \sum_{j=0}^{n}M\mathbb{O}_{j}=M\mathbb{O}_{n+1}-(\alpha+n\beta). $
			\item $ \sum_{j=0}^{n}m\mathbb{O}_{j}=m\mathbb{O}_{n+1}-(\alpha-n\beta). $
		\end{enumerate}
	\end{theorem}
	\section{Conclusion}
	In our study, we have defined the octonions involving the $k$-Mersenne and $k$-Mersenne-Lucas sequence and we have obtained the closed form formulas of these octonions. Moreover, we have presented various results including norm, generating functions, Catalan's identity, Cassini's identity, d'Ocagne's identity, Vajda's identity, and the finite sum formula of these octonions. As a consequence $k=1$ yields the above properties for Mersenne and Mersenne-Lucas octonions.
	\subsection*{Acknowledgment}      
	The first and second authors acknowledge the University Grant Commission(UGC), India for providing fellowship for this research work.
	

\begin{thebibliography}{10}
		
		\bibitem{akkus2015split}
		{\sc Akkus, I., and Ke{\c{c}}ilioglu, O.}
		\newblock Split \uppercase{F}ibonacci and \uppercase{L}ucas octonions.
		\newblock {\em Adv. Appl. Clifford Algebras 25}, 3 (2015), 517--525.
		
		\bibitem{baez2002octonions}
		{\sc Baez, J.}
		\newblock The octonions.
		\newblock {\em Bulletin of the american mathematical society 39}, 2 (2002),
		145--205.
		
		\bibitem{catarino2016modified}
		{\sc Catarino, P.}
		\newblock The modified \uppercase{P}ell and the modified k-\uppercase{P}ell
		quaternions and octonions.
		\newblock {\em Advances in Applied Clifford Algebras 26}, 2 (2016), 577--590.
		
		\bibitem{catarino2016mersenne}
		{\sc Catarino, P., Campos, H., and Vasco, P.}
		\newblock On the \uppercase{M}ersenne sequence.
		\newblock {\em Annales Mathematicae et Informaticae 46\/} (2016), 37--53.
		
		\bibitem{chelgham2021k}
		{\sc Chelgham, M., and Boussayoud, A.}
		\newblock On the k-\uppercase{M}ersenne--\uppercase{L}ucas numbers.
		\newblock {\em Notes on Number Theory and Discrete Mathematics 1}, 27 (2021),
		7--13.
		
		\bibitem{ccimen2016pell}
		{\sc {\c{C}}imen, C.~B., and {\.I}pek, A.}
		\newblock On \uppercase{P}ell quaternions and
		\uppercase{P}ell-\uppercase{L}ucas quaternions.
		\newblock {\em Advances in Applied Clifford Algebras 26}, 1 (2016), 39--51.
		
		\bibitem{ccimen2017jacobsthal}
		{\sc {\c{C}}imen, C.~B., and Ipek, A.}
		\newblock On \uppercase{J}acobsthal and
		\uppercase{J}acobsthal--\uppercase{L}ucas octonions.
		\newblock {\em Mediterranean Journal of Mathematics 14}, 2 (2017), 1--13.
		
		\bibitem{conway2003quaternions}
		{\sc Conway, J.~H., and Smith, D.~A.}
		\newblock {\em On quaternions and octonions: their geometry, arithmetic, and
			symmetry}.
		\newblock AK Peters/CRC Press, 2003.
		
		\bibitem{dasdemir2019gaussian}
		{\sc Dasdemir, A., and Bilgici, G.}
		\newblock Gaussian \uppercase{M}ersenne numbers and \uppercase{M} mersenne
		quaternions.
		\newblock {\em Notes on Number Theory and Discrete Mathematics 25}, 3 (2019),
		87--96.
		
		\bibitem{malini2021mer}
		{\sc Devi, B.~M., and Devibala, S.}
		\newblock On \uppercase{M}ersenne and \uppercase{M}ersenne-\uppercase{L}ucas
		quaternions and octonions.
		\newblock {\em Turkish Online Journal of Qualitative Inquiry 12}, 7 (2021),
		6322--6331.
		
		\bibitem{frontczak2020mersenne}
		{\sc Frontczak, R., and Goy, T.}
		\newblock \uppercase{M}ersenne-\uppercase{H}oradam identities using generating
		functions.
		\newblock {\em Carpathian Mathematical Publications 12}, 1 (2020), 34--45.
		
		\bibitem{godase2019hyperbolic}
		{\sc Godase, A.}
		\newblock Hyperbolic k-Fibonacci and k-Lucas octonions.
		\newblock {\em Notes on number theory and discrete mathematics 26}, 3 (2019),
		176--188.
		
		\bibitem{horadam1963complex}
		{\sc Horadam, A.~F.}
		\newblock Complex \uppercase{F}ibonacci numbers and \uppercase{F}ibonacci
		quaternions.
		\newblock {\em The American Mathematical Monthly 70}, 3 (1963), 289--291.
		
		\bibitem{KarataHalici}
		{\sc Karataş, A., and Halici, S.}
		\newblock Horadam octonions.
		\newblock {\em Analele ştiinţifice ale Universităţii "Ovidius" Constanţa.
			Seria Matematică 25}, 3 (2017), 97--106.
		
		\bibitem{kumari2021some}
		{\sc Kumari, M., Tanti, J., and Prasad, K.}
		\newblock On some new families of k-\uppercase{M}ersenne and generalized
		k-gaussian \uppercase{M}ersenne numbers and their polynomials.
		\newblock {\em arXiv preprint arXiv:2111.09592\/} (2021).
		
		\bibitem{ozkan2022jacobsthal}
		{\sc {\"O}zkan, E., and Uysal, M.}
		\newblock On hyperbolic k-\uppercase{J}acobsthal and
		k-\uppercase{J}acobsthal–\uppercase{L}ucas octonions.
		\newblock {\em Notes on Number Theory and Discrete Mathematics 28}, 02 (2022),
		318--330.
		
		\bibitem{saba2021mersenne}
		{\sc Saba, N., Boussayoud, A., and Kanuri, K.}
		\newblock \uppercase{M}ersenne \uppercase{L}ucas numbers and complete
		homogeneous symmetric functions.
		\newblock {\em Journal of mathematics and computer science 24}, 2 (2021),
		127--139.
		
		\bibitem{Soykan_2021}
		{\sc Soykan, Y.}
		\newblock A study on generalized \uppercase{M}ersenne numbers.
		\newblock {\em Journal of Progressive Research in Mathematics 18}, 3 (Sep.
		2021), 90--108.
		
		\bibitem{szynal2016note}
		{\sc Szynal-Liana, A., and W{\l}och, I.}
		\newblock A note on \uppercase{J}acobsthal quaternions.
		\newblock {\em Advances in Applied Clifford Algebras 26}, 1 (2016), 441--447.
		
		\bibitem{szynal2016pell}
		{\sc Szynal-Liana, A., and W{\l}och, I.}
		\newblock The \uppercase{P}ell quaternions and the \uppercase{P}ell octonions.
		\newblock {\em Advances in Applied Clifford Algebras 26}, 1 (2016), 435--440.
		
		\bibitem{tasci2017k}
		{\sc Tasci, D.}
		\newblock On k-\uppercase{J}acobsthal and
		k-\uppercase{J}acobsthal-\uppercase{L}ucas quaternions.
		\newblock {\em Journal of Science and Arts 17}, 3 (2017), 469--476.
		
		\bibitem{tian2000matrix}
		{\sc Tian, Y.}
		\newblock Matrix representations of octonions and their applications.
		\newblock {\em Advances in Applied Clifford Algebras 10}, 1 (2000), 61--90.
		
		\bibitem{uslu2017some}
		{\sc Uslu, K., and Deniz, V.}
		\newblock Some identities of k-\uppercase{M}ersenne numbers.
		\newblock {\em Advances and Applications in Discrete Mathematics 18}, 4 (2017),
		413--423.
		
	\end{thebibliography}
\end{document}